\documentclass[10pt,a4paper]{article}

\usepackage{amsfonts}
\usepackage{amssymb}
\usepackage{amsthm}
\usepackage[english]{babel}
\usepackage{hhline}
\usepackage[utf8]{inputenc}
\usepackage{amsmath}
\usepackage[all,2cell,ps]{xy}
\usepackage{qsymbols}
\usepackage{color}
\usepackage{epsfig}
\usepackage{color}
\usepackage{graphics}
\usepackage{graphicx}%
\usepackage{enumerate}
\usepackage{tikz}
\usepackage[bottom]{footmisc}
\usepackage{authblk}

\usepackage[colorlinks,linkcolor=black,citecolor=blue,urlcolor=blue]{hyperref}
\theoremstyle{plain}

\newtheorem{theorem}{Theorem}
\newtheorem{lemma}[theorem]{Lemma}
\newtheorem{proposition}[theorem]{Proposition}

\newtheorem{definition}[theorem]{Definition}
\newtheorem{example}[theorem]{Example}
\newtheorem{remark}[theorem]{Remark}

\newcommand{\hIL}{\mathsf{hIL}}
\newcommand{\SRL}{\mathsf{SRL}}
\newcommand{\KhIL}{\mathsf{KhIL}}
\newcommand{\KhILc}{\mathsf{KhIL}^{c}}
\newcommand{\NA}{\mathsf{NA}}
\newcommand{\SNA}{\mathsf{SNA}}
\newcommand{\SH}{\mathsf{SH}}
\newcommand{\SN}{\mathsf{SN}}

\newcommand{\ce}{\mathsf{c}}
\newcommand{\CK}{\mathsf{(CK)}}
\newcommand{\CC}{\mathsf{(C)}}
\newcommand{\C}{\mathsf{C}}
\newcommand{\K}{\mathrm{K}}
\newcommand{\Con}{\mathrm{Con}}
\newcommand{\co}{\theta}
\newcommand{\ra}{\rightarrow}
\newcommand{\rla}{\leftrightarrow}
\newcommand{\Ra}{\Rightarrow}
\newcommand{\we}{\wedge}

\newcommand{\s}{\sim}
\newcommand{\gIF}{\mathrm{hIF}}

\title{Hemi-Nelson algebras}

\author{Noem\'i Lubomirsky, Paula Mench\'on and Hern\'an Javier San Mart\'in}

\date{}

\begin{document}

\maketitle

\begin{abstract}
The aim of this paper is to generalize the link between Heyting algebras and Nelson algebras, 
established independently by Fidel and Vakarelov at the end of the 1970s, 
in the framework of bounded distributive hemi-implicative lattices.
For this purpose, we introduce the variety of hemi-Nelson algebras. 
Moreover, we characterize the lattice of congruences of a hemi-Nelson 
algebra in terms of certain implicative filters. We also esta\-blish an equivalence 
between the algebraic category of bounded distributive hemi-implicative lattices and 
the one of centered hemi-Nelson algebras.
\end{abstract}

\section{Introduction}

In this paper we focus on unifying ideas derived from several algebraic varieties connected to intuitionistic logics,
as for instance Heyting algebras \cite{BD}, semi-Heyting algebras \cite{S1}, subresiduated lattices \cite{EH},
Nelson algebras \cite{V}, semi-Nelson algebras \cite{CV} and subresiduated Nelson algebras \cite{LMSM}.

Hemi-implicative lattices \cite{CFMSM0} are defined as lattices with a greatest element, which will be denoted by $1$,
endowed with a binary operation $\ra$ satisfying the equation $x\ra x = 1$ and the inequality $x\we (x\ra y)\leq y$.
The previous inequality can be replaced by the following quasi-equation:
\[
\text{If}\; z\leq x\ra y\;\text{then}\; z\we x\leq y.
\] 
A bounded distributive hemi-implicative lattice is defined as a
hemi-implicative lattice whose underlying lattice is distributive and
has a least element
\footnote{Bounded distributive hemi-implicative lattices were introduced and studied in \cite{SM2} under
the name of weak implicative lattices.}.
Many varieties of interest for algebraic logic are subvarieties of the variety of 
bounded distributive hemi-implicative lattices. Some relevant examples of these
algebras are Heyting algebras, subresiduated lattices, RWH-algebras \cite{CJ}, semi-Heyting algebras 
and Hilbert bounded distributive lattices (i.e., Hilbert algebras \cite{CCM,D} whose
natural order defines a bounded distributive lattice). 

Nelson's constructive logic with strong negation, which was introduced and studied in \cite{Nel} (see also \cite{R,S,V}),
is a non-classical logic that combines the constructive approach of positive intuitionistic logic with a classical (i.e. De Morgan) negation.
The algebraic models of this logic are called Nelson algebras. The class of Nelson algebras,
which is a variety, has been studied since at least the late 1950's (firstly by Rasiowa; see \cite{R} and
references therein). In the the end of the 1970's it was proved, independently by Fidel and Vakarelov, 
that every Nelson algebra can be represented as a special binary product (here called a twist structure) 
of a Heyting algebra. 

The main goal of the present manuscript is to extend the twist construction in the framework of 
bounded distributive hemi-implicative lattices, thus obtai\-ning a new variety, whose members will be called 
hemi-Nelson algebras. More precisely, we show that every hemi-Nelson algebra
can be represented as a twist structure of a bounded distributive hemi-implicative lattice. 

Hemi-Nelson algebras are an extension of Kleene algebras, together with a hemi-implication. These structures serve as an algebraic framework for non-classical logics, specifically those that seek to balance the properties of constructive logic with the involutive negation found in Kleene's systems. In this context, the logic of Kleene algebras was thoroughly studied and characterized in \cite{KB}, where a formalization of its properties is provided.

The paper is organized as follows. In Section \ref{s1} we recall the definition of Nelson algebra and 
sketch the main constructions linking Heyting algebras with Nelson algebras. Moreover,
we recall the definition of bounded distributive hemi-implicative lattice and we give some of its properties. 
In Section \ref{s2} we introduce hemi-Nelson algebras
and show that every hemi-Nelson algebra can be represented as a twist structure over a 
bounded distributive hemi-implicative lattice.
In Section \ref{s3} we prove that given an arbitrary hemi-Nelson algebra, there exists an order isomorphism
between the lattice of its congruences and the lattice of its h-implicative filters 
(which are a kind of implicative filters).
In Section \ref{s4} we characterize those hemi-Nelson algebras that 
can be represented as a twist structure of a bounded distributive hemi-implicative lattice 
and we also prove that there exists  an equivalence between the algebraic category
of bounded distributive hemi-implicative lattices and the algebraic category of centered  
hemi-Nelson algebras, where a centered hemi-Nelson algebra is defined as a hemi-Nelson algebra 
endowed with a center, i.e., a fixed element with respect to the involution 
(this element is necessarily unique).  
Finally, in Section \ref{s5} we applied some results developed throughout the present paper
in order to investigate certain aspects of semi-Nelson algebras and subresiduated Nelson algebras respectively.

\section{Basic results}\label{s1}

In this section we recall the definition of Nelson algebra \cite{BC,C,Vig} and
some links between Heyting algebras and Nelson algebras \cite{V}. 
Finally we recall the definition of bounded distributive hemi-implicative lattices and some of their properties \cite{SM2}.

Recall that a \textit{Kleene algebra} \cite{BC,C,K} is a bounded distributive lattice 
endowed with a unary operation $\sim$ which satisfies the following equations: 
\begin{enumerate}[\normalfont Ne1)]
\item $\sim \sim x = x$,
\item $\sim (x\we y) = \sim x \vee \sim y$,
\item $(x\we \sim x)\we (y \vee \sim y) = x\we \sim x$.
\end{enumerate}

\begin{definition}\label{generalised2}
An algebra $\langle T,\we,\vee,\ra,\sim,0,1 \rangle$ of type $(2,2,2,1,0,0)$ is called a 
\textit{Nelson algebra} if $\langle T,\we,\vee,\sim,0,1 \rangle$ is a Kleene algebra and 
for every $x,y,z\in T$ the following equations are satisfied:
\begin{enumerate}[\normalfont Ne4)]
\item $x\ra x = 1$,
\end{enumerate}
\begin{enumerate}[\normalfont Ne5)]
\item $x\ra (y\ra z) = (x\we y)\ra z$,
\end{enumerate}
\begin{enumerate}[\normalfont Ne6)]
\item $x\we (x\ra y) = x \we (\sim x \vee y)$,
\end{enumerate}
\begin{enumerate}[\normalfont Ne7)]
\item $\sim x\vee y \leq x\ra y $, 
\end{enumerate}
\begin{enumerate}[\normalfont Ne8)]
\item $x\ra (y\we z) = (x\ra y)\we (x\ra z)$.
\end{enumerate}
\end{definition}

Let $\langle T,\we,\vee,\ra,\sim,0,1\rangle$ be an algebra of type $(2,2,2,1,0,0)$.
Assume that $\langle T,\we,\vee,0,1\rangle$ is a bounded distributive lattice
and that condition~Ne6) holds. Then conditions Ne7) and Ne8) also hold \cite{Mon1,Mon2}. 
We write $\NA$ for the variety of Nelson algebras.

There are two key constructions that relate Heyting algebras with Nelson algebras.
Given a Heyting algebra $A$, we define the set 
\begin{equation} \label{ka}
\K(A) = \{(a,b)\in A\times A: a\we b = 0\}
\end{equation}
and then endow it with the following operations:
\begin{itemize}
\item $(a,b)\we (c,d) = (a\we c,b\vee d)$,
\item $(a,b)\vee (c,d) = (a\vee c, b\we d)$,
\item $\s (a,b) = (b,a)$,
\item $(a,b)\Ra (c,d) = (a\ra c,a\we d)$,
\item $\perp = (0,1)$, 
\item $\top = (1,0)$.
\end{itemize}
Then $\langle \K(A),\we,\vee,\Ra,\perp,\top \rangle\in \NA$ \cite{V}. In the same manuscript,
Vakarelov proves that if $T\in \NA$, then the relation $\theta$ 
defined by 
\begin{equation}\label{rel}
x \theta y\;\text{if and only if}\; x\ra y = 1\;\text{and}\; y\ra x = 1 
\end{equation}
is an equivalence relation such that $\langle T/\theta,\we,\vee,\ra,0,1 \rangle$ is a Heyting algebra 
with the operations defined by
\begin{itemize}
\item $x/\co \we y/\co := x \we y/\co$,
\item $x/\co \vee y/\co := x \vee y/\co$,
\item $x/\co \ra y/\co := x\ra y/\co$,
\item $0 := 0/\co$,
\item $1 := 1/\co$.
\end{itemize}

Now we recall the definition of bounded distributive hemi-implicative lattice.

\begin{definition}
An algebra $\langle A,\we,\vee,\ra,0,1 \rangle$ of type $(2,2,2,0,0)$ is said to be a \textit{bounded distributive
hemi-implicative lattice} (h-lattice for short) if the algebra $(A,\we,\vee,\ra,0,1)$ is a bounded distributive
lattice and the following two conditions are satisfied for every $a,b \in A$:
\begin{enumerate}[\normalfont 1)]
\item $a\ra a = 1$,
\item $a\we (a\ra b)\leq b$.
\end{enumerate}
\end{definition}

We write $\hIL$ for the variety whose members are h-lattices.
Given $A\in \hIL$ and $a,b \in A$, we define $a\rla b = (a\ra b)\we (b\ra a)$.

\begin{lemma} \label{r1}
Let $A\in \hIL$ and $a,b \in A$. Then the following conditions are satisfied:
\begin{enumerate}[\normalfont 1)]
\item If $a\ra b = 1$, then $a\leq b$,
\item $a\ra b = 1$ and $b\ra a = 1$ if and only if $a = b$,
\item $a\we (a\rla b) = b \we (a\rla b)$.
\end{enumerate}
\end{lemma}

\begin{proof}
It follows from a straightforward computation.
\end{proof}

The following example can be found in \cite[Example 1]{SM2}.

\begin{example}\label{example1}
	The set $A=\lbrace 0,a,1\rbrace$ where the elements are totally ordered $0<a<1$ and the operation $\ra$ is defined by
	\begin{center}
	\begin{tabular}{|c|c|c|c|}
		\hline
		$\ra$& $0$ & $a$ & $1$ \\
		\hline
		$0$ & $1$ & $a$ & $1$ \\
		\hline
		$a$ & $0$ & $1$ & $0$ \\
		\hline
		$1$ & $0$ & $a$ & $1$ \\
		\hline
	\end{tabular}
	\end{center}
	is an h-lattice.
\end{example}

\section{Hemi-Nelson algebras} \label{s2}

In this section we define hemi-Nelson algebras and we show 
that every hemi-Nelson algebra can be represented as a 
twist structure of an h-lattice.

Let $A\in \hIL$. We define $\K(A)$ as in (\ref{ka}). 
Then $\langle \K(A),\we,\vee,\s,(0,1),(1,0) \rangle$ is a Kleene algebra \cite{C,K}.
On $\K(A)$ we also define the binary operation $\Ra$ as in Section \ref{s1}.
Note that this is a well defined map because if $(a,b)$ and $(c,d)$ are elements of $\K(A)$
then $(a\ra c)\we a\we d \leq c\we d = 0$, i.e., $(a\ra c)\we a\we d = 0$.
Thus, the structure $\langle \K(A),\we,\vee,\Ra,\s,(0,1),(1,0) \rangle$ is an algebra of type
$(2,2,2,1,0,0)$.

\begin{definition}\label{generalised}
An algebra $\langle T, \we, \vee, \ra, \s, 0, 1\rangle$ of type $(2, 2, 2, 1, 0, 0)$ is said to be a
\textit{hemi-Nelson algebra} (h-Nelson algebra for short) 
if $\langle T, \wedge, \vee, \sim, 0, 1\rangle$ is a Kleene algebra and the
following conditions are satisfied for every $x,y,z\in T$:
\begin{enumerate} [\normalfont (hN1)]
\item $x\ra x=1$, \label{e1}
\item $x\we (x\ra y)\leq x\we (\s x\vee y)$, \label{e2}
\item $\s(x\ra y)\ra (x\we \s y)=1$, \label{e3}
\item $(x\we \s y)\ra \s(x\ra y)=1$, \label{e4}
\item $(x\we y \we (x\ra y)) \ra (x \we (x\ra y)) = 1$, \label{e5}
\item $(x\we (x\ra y)) \ra (x\we y \we (x\ra y)) = 1$, \label{e6}

\item if $x\ra y = 1$, $y\ra x = 1$, $y\ra z = 1$ and $z\ra y = 1$ then $x\ra z = 1$ and $z\ra x = 1$, \label{trans}
\item if $x\ra y = 1$ and $y\ra x = 1$ then $(x\we z)\ra (y\we z) = 1$,\label{infimo}
\item if $x\ra y = 1$ and $y\ra x = 1$ then $(x\vee z)\ra (y\vee z) = 1$,\label{supremo}
\item if $x\ra y = 1$ and $y\ra x = 1$ then $(x\ra z) \ra (y\ra z) = 1$ and $(z\ra x) \ra (z\ra y) = 1$. \label{implic}
\end{enumerate}
\end{definition}

We write $\KhIL$ for the class whose members are h-Nelson algebras.

\begin{proposition} \label{Kalman}
If $A\in \hIL$ then $\K(A)\in \KhIL$.
\end{proposition}

\begin{proof}
Let $A\in \hIL$. We will show that $\K(A)$ satisfies the conditions of Definition
\ref{generalised2}. In order to prove it, let $x=(a,b)$, $y = (c,d)$ and $z = (e,f)$ be elements of $\K(A)$.

It is immediate that $x\ra x = (1,0)$, so (hN1) 
is satisfied.

In order to show (hN2),  
note that
\[
x\we (x\ra y) = (a\we (a\ra c), b \vee (a\we d)),
\]
\[
x\we (\s x \vee y) = (a\we (b\vee c), b\vee (a\we d)).
\]
Since $a\we (a\ra c) \leq a\we c \leq a\we (b\vee c)$ then
\[
x\we (x\ra y)\leq x\we (\s x\vee y),
\]
so (hN2) is satisfied.

A direct computation shows that
\[
\s(x\ra y)\ra (x\we \s y) = (1,a\we d \we (b\vee c)).
\]
But 
\[
a\we d \we (b\vee c) = (a\we d \we b) \vee (a\we d \we c) = 0,
\]
so 
\[
\s(x\ra y)\ra (x\we \s y) = (1,0),
\]
which is (hN3).

Condition (hN4)
follows from the equality
\[
(x\we \s y)\ra \s(x\ra y) = (1, a\we d\we (a\ra c))
\]
and the fact that $a\we d\we (a\ra c) \leq c\we d = 0$, i.e.,
$a\we d\we (a\ra c) = 0$.

Now we will show (hN5) and (hN6).
Straightforward computations
show that
\begin{equation} \label{a1}
x\we y\we (x\ra y) = (a\we c\we (a\ra c),b\vee d),
\end{equation}
\begin{equation} \label{a2}
x\we (x\ra y) = (a\we (a\ra c), b\vee (a\we d)).
\end{equation}
Besides, 
$a\we (a\ra c)\leq a\we c\we (a\ra c)$, so
\begin{equation} \label{a3}
a\we c\we (a\ra c) = a\we (a\ra c).
\end{equation} 
Moreover, taking into account the distributivity of the underlying 
lattice of $A$ and the facts that $a\we b = c\we d = 0$ and $a\we (a\ra c)\leq c$, 
we get
\begin{equation} \label{a4}
(a\we c\we (a\ra c))\we (b\vee (a\we d)) = 0,
\end{equation}
\begin{equation} \label{a5}
(a\we (a\ra c))\we (b\vee d) = 0.
\end{equation}
Thus, a direct computation based on equations (\ref{a1}),\dots,(\ref{a5}) shows that conditions
(hN5) and (hN6)
are satisfied.

Suppose that $x\ra y = 1$, $y\ra x = 1$, $y\ra z = 1$ and $z\ra y = 1$, so
$a\ra c = 1$, $a\we d = 0$, $c\ra a = 1$, $c\we b = 0$, $c\ra e = 1$,
$c\we f = 0$, $e\ra c = 1$ and $e\we d = 0$. In particular, it follows
from Lemma \ref{r1} that $a = c = e$, so
\[
x \ra z = (1,a\we f).
\]
However, $a\we f =  e\we f = 0$.
Hence, $x\ra z = (1,0)$. Analogously, we have that $z\ra x = (1,0)$.
Thus, (hN7) 
is satisfied.

Finally, we will show (hN8) and (hN9).
Suppose that $x\ra y = 1$ and $y\ra x = 1$,
so $a = c$, $a\we d = 0$ and $c\we b = 0$. Besides,
\[
x\we z = (a\we e,b\vee f),
\]
\[
y\we z = (c\we e,d\vee f) = (a\we e,d\vee f),
\]
which implies that
\[
(x\we z) \ra (y\we z) = (1,a\we e\we (d\vee f)).
\]
Since $a\we e\we (d\vee f) = (a\we e\we d) \vee (a\we e \we f)$ and $a = c$ then
$a\we e\we (d\vee f) = 0$, so
\[
(x\we z) \ra (y\we z) = (1,0),
\]
which is (hN8).
In order to see (hN9),
note that
\[
x\vee z = (a\vee e,b\we f),
\]
\[
(c\vee e,d\we f) = (a\vee e,d\we f),
\]
so 
\[
(x\vee z) \ra (y\vee z) = (1,(a\vee e)\we d\we f).
\]
Taking into account that $(a\vee e)\we d\we f = (a \we d \we f) \vee (e\we d\we f)$,
$a = c$, $a\we c = 0$ and $e\we f = 0$ we get 
\[
(a\vee e)\we d\we f = 0. 
\]
Therefore,
\[
(x\vee z) \ra (y\vee z) = (1,0),
\]
which was our aim.
\end{proof}

\begin{proposition} \label{pvi}
Let $T\in \KhIL$.
The following conditions are satisfied for every $x,y\in T$:
\begin{enumerate}[\normalfont 1)]
\item $1\ra x\leq x$,
\item if $x\ra y = 1$ then $x = x\we (\s x \vee y)$,
\item if $x\ra y = 1$ and $\s y\ra \s x = 1$ then $x\leq y$. 
\end{enumerate}
\end{proposition}

\begin{proof}
First we will show that condition 1) is satisfied. By (hN2), 
\[
(1\ra x)=1\we (1\ra x)\leq 1\we(\s 1 \vee x)=x, i.e.,
\]
$1\ra x \leq x$.

Now we will show that 2) is verified. Suppose that $x\ra y = 1$, so by (hN2)
we get $x\leq (\s x\vee y)$, i.e., 
\[
x = x\we (\s x \vee y). 
\]

Finally we will see 3) is satisfied. Assume that $x\ra y = 1$ and $\s y\ra \s x = 1$.
It follows from 2. of this proposition that
\[
x = x\we (\s x \vee y) = (x\we \s x) \vee (x\we y),
\]
\[
\s y = \s y \we (y \vee \s x) =(\s y\we y)\vee (\s y \we \s x),
\]
so
\[
y = (y\vee \s y) \we (x\vee y).
\]
In particular,
\[
x =(x\we \s x)\vee (x\we y) \leq (y\vee \s y) \vee (x\we y) = y\vee \s y,
\]
so
\[
x = x\we (x\vee y) \leq (y\vee \s y) \we (x\vee y) = y.
\]
Therefore, $x\leq y$.
\end{proof}

\begin{example}\label{example2}
Considering the h-lattice given in Example \ref{example1} 
we have that $\K(A)$ is endowed by a structure of h-Nelson algebra, 
where its universe is the chain of five elements given by $(0,1)<(0,a)<(0,0)<(a,0)<(1,0)$.
\end{example}

Let $T\in \KhIL$. We define the binary relation $\co$ as in (\ref{rel}) of Section \ref{s1}.

\begin{lemma} \label{l1}
Let $T\in \KhIL$. Then $\co$ is an equivalence relation compatible with $\we$, $\vee$ and $\ra$.
\end{lemma}

\begin{proof}
The reflexivity and symmetry of $\co$ are immediate. The transitivity of the relation
follows from (hN7).
Thus, $\co$ is an equivalence relation.
The facts that $\co$ is compatible with $\we$, $\vee$ and $\ra$ follow from (hN8), (hN9) and (hN10).
\end{proof}

Let $\langle T, \we, \vee, \ra, \s, 0, 1\rangle$ be an h-Nelson algebra.
Then it follows from Lemma \ref{l1} that we can define on $T/\theta$ the operations 
$\we$, $\vee$, $\ra$, $0$ and $1$ as in Section \ref{s1}.
In particular, $\langle T/\co, \we, \vee, 0, 1\rangle$ is a bounded distributive lattice.
We denote by $\preceq$ to the order relation of this lattice.

\begin{proposition}
If $T\in \KhIL$ then $T/\co \in \hIL$.
\end{proposition}

\begin{proof}
Let $T\in \KhIL$ and $x,y\in T$. It follows from (hN1) 
that 
\[
x/\theta \ra x/\theta = 1/\theta. 
\]
The inequality 
\[
x/\theta \we (x/\theta \ra y/\theta) \preceq y/\theta
\] 
follows from (hN5) and (hN6).
Therefore, $T/\co \in \hIL$.
\end{proof}

\begin{theorem}\label{rept}
Let $T\in \KhIL$. Then the map $\rho_T:T\ra \K(T/\co)$ given by $\rho_{T}(x) = (x/\co,\s x/\co)$ is a monomorphism.
\end{theorem}

\begin{proof}
We write $\rho$ instead of $\rho_T$.
First we will show that $\rho$ is a well defined map. Let $x\in T$.
Note that $x/\co \we \s x/\co = 0/\co$ if and only if $(x\we \s x) \ra 0 = 1$
and $0\ra (x\we \s x) = 1$. First note that it follows from (hN1) and (hN4)
that
\[
(x\we \s x) \ra 0 = (x\we \s x) \ra \s (x\ra x) = 1.
\]
Besides, it follows from (hN1) and (hN3)
that
\[
0\ra (x\we \s x) = \s (x\ra x) \ra (x\we \s x) = 1.
\]
Thus, $\rho$ is a well defined map.

Straightforward computations show that $\rho$ preserves $\we$, $\vee$, $\s$, $0$ and $1$.
The fact that $\rho$ preserves $\ra$ follows from (hN3) and (hN4).
Hence, $\rho$ is a homomorphism. 
Finally, in order to show that $\rho$ is an injective map let $x,y \in T$ such
that $\rho(x) = \rho(y)$, i.e., $x\ra y= 1$, $y\ra x = 1$, $\s x \ra \s y = 1$ and $\s y \ra \s x = 1$.
Thus, it follows from Lemma \ref{pvi} that $x = y$. Hence, $\rho$ is an injective map.
Therefore, $\rho$ is a monomorphism.
\end{proof}

It is natural to ask whether the quasivariety 
$\KhIL$ is in fact a variety. The answer is positive, and to prove it we first establish the following results.

\begin{proposition} \label{eq neg}
    Let $T\in \KhIL$. Then, the following condition is satisfied for all $x,y,z\in T$:
    \begin{equation}\label{eq n}
        (x\ra y)\vee (z\we\s z)=(x\vee (z\we \s z))\ra (y\vee (z\we \s z)).
    \end{equation}
\end{proposition}
\begin{proof}
   Let $A\in \hIL$. We show that equation~\eqref{eq n} holds in $\K(A)$. 
Take $x=(a,b)$, $y=(c,d)$ and $z=(e,f)$ in $\K(A)$. Then
\[
(x\ra y)\vee (z\we \s z)
  = (a\ra c,\,(a\we d)\we (e\vee f)),
\]
and
\[
(x\vee (z\we \s z))\ra (y\vee (z\we \s z))
  = (a,\, b\we (e\vee f))\ra (c,\, d\we (e\vee f)).
\]
It is straightforward to verify that
\[
(a,\, b\we (e\vee f))\ra (c,\, d\we (e\vee f))
   = (a\ra c,\,(a\we d)\we (e\vee f)).
\]
Thus equation~\eqref{eq n} holds in $\K(A)$.

Finally, by Theorem~\ref{rept}, every $T\in \KhIL$ is isomorphic to a subalgebra of 
$\K(A)$ for some $A\in \hIL$, and hence equation~\eqref{eq n} holds in $T$ as well.
\end{proof}

\begin{lemma}\label{kleene cond}
    Let $T$ be a Kleene algebra. Then,
    \[1=x\vee (z\we\s z) \text{ if and only if }x=1.\]
\end{lemma}
\begin{proof}
Suppose that $1 = x \vee (z \we \s z)$. It follows that $1\leq x \vee \s z$, and by De Morgan’s laws we obtain
    $\s x \we z = 0$.  
    Since $z \we \s z \leq x \vee \s x$, we have $x \vee (z \we \s z) \leq x \vee \s x$, which implies $x \vee \s x = 1$.  
    Therefore,
    \[
    z = z \we (x \vee \s x)
      = (z \we x) \vee (z \we \s x)
      = (z \we x) \vee 0
      = z \we x.
    \]
    Hence $z \leq x$, and consequently $1 = x \vee z = x$. 
    The converse implication is immediate and the result follows.
\end{proof}

In \cite{J} and \cite{Vig}, for a Kleene algebra $T$ and its set of negative elements $T^-=\{\,x\in T : x\leq \s x\,\}$, the following relation is introduced:
\begin{equation}\label{rel1}
x\,\theta_{T^-}\,y \quad\text{if and only if}\quad \exists\, n\in T^- \;\text{such that}\; x\vee n = y\vee n.
\end{equation}
Since $T^-$ is an ideal of $T$, the relation $\theta_{T^-}$ is a lattice congruence. Moreover, the quotient lattice $T/\theta_{T^-}$ is distributive, and it provides 
the underlying distributive structure employed in the twist-product 
representation of Kleene algebras.  
The motivation behind equation \eqref{eq n} is that 
it ensures the compatibility of $\theta_{T^-}$ with the operation~$\ra$. Under this assumption, it will be shown that $\theta=\theta_{T^-}$. 
Note that $T^-=\{\,x\wedge \s x : x\in T\,\}$, a description that will be used in the sequel.

\begin{proposition} \label{equivalence eq}
    If $T$ be a Kleene algebra satisfying Conditions from (hN1) to (hN6) and equation (\ref{eq n}), the following conditions are equivalent:
    \begin{enumerate}
        \item $x\ra y=1$ and $y\ra x=1$;
        \item There exists $z\in T$ such that $x\vee (z\we\s z)=y\vee (z\we\s z)$.
    \end{enumerate}
\end{proposition}

\begin{proof}
  Suppose that $x \ra y = 1$ and $y \ra x = 1$.  
From Proposition~\ref{pvi} we obtain $x = x \we (\s x \vee y)$ and
$y = y \we (\s y \vee x)$.
Let $z = (x \we \s x) \vee (y \we \s y)$.  
Then $z \leq \s z$ and $z \we \s z = z$.  
Hence,
\[ x \vee z
   = x \vee (x \we \s x) \vee (y \we \s y)
   = x \vee (y \we \s y)
   = (x \vee y) \we (x \vee \s y).
\]
Since $y \leq \s y \vee x$, we have
$x \vee y \leq \s y \vee x$,
and therefore $x \vee (z \we \s z) = x \vee y$.
By the same reasoning, $y \vee (z \we \s z) = x \vee y$,
and thus, $y \vee (z \we \s z) = x \vee (z \we \s z)$.

Now, suppose that there exists $z\in T$ such that $x\vee (z\we\s z)=y\vee (z\we\s z)$. By (hN1) and (\ref{eq n}), $1=x\vee (z\we\s z)\ra y\vee (z\we\s z)=(x\ra y)\vee (z\we\s z)$. By Lemma \ref{kleene cond}, $x\ra y=1$. By the same reasoning we obtain $y\ra x=1$.
\end{proof}

\begin{theorem}
    The class $\KhIL$ is a variety.
\end{theorem}
\begin{proof}
We will prove that $\KhIL$ is precisely the class of algebras defined by the equations of Kleene algebras together with conditions (hN1)–(hN6) and equation (\ref{eq n}). 

On the one hand, from Proposition~\ref{eq neg} it follows that every algebra in 
$\KhIL$ satisfies equation (\ref{eq n}).

On the other hand, assume that $T$ satisfies conditions~(hN1)--(hN6) and equation (\ref{eq n}). From Proposition \ref{equivalence eq} it follows immediately that $T$ satisfies conditions (hN8), (hN9) and (hN7), since $\theta_{T^-}$ is a lattice congruence of $T$.

To prove condition (hN10), suppose that $x\ra y=y\ra x=1$. By Proposition \ref{equivalence eq}, there exists $w\in T$ such that $x\vee (w\we\s w)=y\vee (w\we\s w)$. Then, 
\[z\vee (w\we\s w)\ra x\vee (w\we\s w)=z\vee (w\we\s w)\ra y\vee (w\we\s w)\]
and by equation (\ref{eq n}), it follows that
\[(z\ra x)\vee (w\we\s w)=(z\ra y)\vee (w\we\s w).\]
Again, by Proposition \ref{equivalence eq}, $(z\ra x)\ra (z\ra y)=1$ and $(z\ra y)\ra (z\ra x)=1$. Analogously, it is proved that $(y\ra z)\ra (x\ra z)=1$ and $(x\ra z)\ra (y\ra z)=1$. Therefore, $T\in \KhIL$ and the result follows.
\end{proof}

\begin{remark}
    Equation (\ref{eq n}) is independent from Conditions (hN1)-(hN6). To see this, let us consider the chain $T=\{0,a,1\}$ where $0<a<1$ and the operation $\ra$ is defined by
	\begin{center}
	\begin{tabular}{|c|c|c|c|}
		\hline
		$\ra$& $0$ & $a$ & $1$ \\
		\hline
		$0$ & $1$ & $1$ & $1$ \\
		\hline
		$a$ & $1$ & $1$ & $a$ \\
		\hline
		$1$ & $0$ & $a$ & $1$ \\
		\hline
	\end{tabular}
	\end{center}
    and $\s 0=1$, $\s 1=0$ and $\s a=a$. $T$ is a Kleene algebra that satisfies Conditions (hN1)-(hN6) but
    \[(0\ra 1)\vee a=1\vee a=1\neq (0\vee a\ra 1\vee a)=(a\ra 1)=a.\]
\end{remark}

\section{Congruences of h-Nelson algebras}\label{s3}

Let $T\in \KhIL$ and $F\subseteq T$. 
As usual, we say that $F$ is a \textit{filter} of $T$ if $1\in F$, $F$ is an upset (i.e.,
for every $x, y \in T$, if $x\leq y$ and $x\in F$ then $y\in F$) and $x\we y \in F$, for all $x, y\in F$.
We also say that $F$ is an \textit{implicative filter} of $T$
if $1\in F$ and for every $x,y \in F$, if $x\in F$ and $x\ra y \in F$ then $y\in F$.

\begin{definition}\label{gf}
Let $T\in \KhIL$ and $F\subseteq T$. 
We say that $F$ is an \textit{h-implicative filter} of $T$ if it is an implicative filter of $T$ which
satisfies the following additional conditions for every $x,y\in T$ and $f\in F$:
\begin{enumerate}[\normalfont F1)]
\item $(x\ra y)\ra ((x\we f)\ra (y\we f))\in F$,
\item $((x\we f)\ra (y\we f)) \ra (x\ra y)\in F$,
\item $\s (x\ra y)\ra \s ((x\we f)\ra (y\we f))\in F$,
\item $\s ((x\we f)\ra (y\we f)) \ra \s (x\ra y)\in F$.
\end{enumerate}
\end{definition}

In this section we prove that for every $T\in \KhIL$ there exists an order isomorphism
between the lattice of congruences of $T$ and the lattice of h-implicative filters of $T$.
In order to make it possible, we proceed to give some definitions and technical results.

Let $T\in \KhIL$. 
Given $\theta$ an equivalence relation on $T$ and $x\in T$,
we write $x/\theta$ for the equivalence class of $x$ associated to the congruence $\theta$.
We also write $T/\theta$ for the set $\{x/\theta:x\in T\}$.
Moreover, we write $\Con(T)$ to indicate the set of congruences of $T$.
Finally, we write $\gIF(T)$ for the set of h-implicative filters of $T$. 

\begin{remark}\label{remark18}
Let $T\in \KhIL$ and $F\in \gIF(T)$. Then $F$ is an upset.
Indeed, let $x, y\in T$ such that $x\leq y$ and $x\in F$.
It follows from F2) that $((x\we x)\ra (y\we x)) \ra (x\ra y)\in F$,
i.e., $1\ra (x\ra y) \in F$. Since $1\in F$ then $x\ra y \in F$.
Taking into account that $x\in F$ we get $y\in F$. Therefore, $F$ is an upset.
\end{remark}

\begin{lemma} \label{lc2}
Let $T\in \KhIL$ and $\theta \in \Con(T)$. Then $1/\theta \in \gIF(T)$.
\end{lemma}

\begin{proof}
Let $\theta \in \Con(T)$. Note that $1/\theta$ is an upset because
$\theta$ is in particular a lattice congruence. In order to show that $1/\theta$ is an implicative 
filter, let $x,y\in T$ such that $x, x\ra y\in 1/\theta$. Then $x \we (x\ra y)\in 1/\theta$.
Since $1/\theta$ is an upset then it follows from (hN2)
that $x \we (\s x \vee y) \in 1/\theta$.
Moreover, since $x/\theta = 1/\theta$ then $\s x/\theta = 0/\theta$, so $y/\theta = 1/\theta$, i.e.,
$y\in 1/\theta$. Hence, $1/\theta$ is an implicative filter. The fact that $1/\theta$ is an 
h-implicative filter follows from a direct computation.
\end{proof}

\begin{lemma} \label{lc3}
Let $T\in \KhIL$, $\theta \in \Con(T)$ and $x,y \in T$. Then $(x,y)\in \theta$ if and only if
$x\ra y$, $\s y \ra \s x$, $y\ra x$, $\s x \ra \s y \in 1/\theta$.
\end{lemma}

\begin{proof}
Let $\theta \in \Con(T)$ and $x,y \in T$. It is immediate that if $(x,y)\in \theta$ then
$x\ra y$, $\s y \ra \s x$, $y\ra x$, $\s x \ra \s y \in 1/\theta$.
Conversely, assume that $x\ra y$, $\s y \ra \s x$, $y\ra x$, $\s x \ra \s y \in 1/\theta$.
Since $T/\theta \in \KhIL$ then it follows from Proposition \ref{pvi} that
$x/\theta = y/\theta$, i.e., $(x,y)\in \theta$.
\end{proof}

\begin{lemma} \label{lc4}
Let $T\in \KhIL$ and $\theta,\psi \in \Con(T)$. Then $\theta\subseteq \psi$ if and only if
$1/\theta \subseteq 1/\psi$. In particular, $\theta = \psi$ if and only if $1/\theta = 1/\psi$.
\end{lemma}

\begin{proof}
Let $\theta,\psi \in \Con(T)$. It is immediate that if $\theta\subseteq \psi$ then
$1/\theta \subseteq 1/\psi$. The converse follows from a direct computation based in Lemma \ref{lc3}.
\end{proof}

\begin{lemma}\label{lc5}
Let $T\in \KhIL$ and $F$ an implicative filter of $T$ which satisfies
F1) and F2) of Definition \ref{gf}
for every $x,y\in T$ and $f\in F$. Then $F$ is a filter. Moreover, $x\in F$
if and only if $1\ra x\in F$.
\end{lemma}

\begin{proof}
 It follows from the definition of implicative filter that $1\in F$. From F2), in Remark \ref{remark18}, it follows that $F$ is an upset.
Now, let $x\in T$. In order to show that $x\in F$ if and only if $1\ra x\in F$,
suppose that $x\in F$. Then
\[
1\ra (1\ra x) = ((1\we x)\ra (x\we x))\ra (1\ra x)\in F,
\]
so $1\ra x\in F$. The converse follows from Proposition \ref{pvi}.

Finally, in order to show that $F$ is closed by $\wedge$, let
$x, y \in F$.
Then
$1\ra y\in F$. Moreover,
\[
(1\ra y) \ra ((1\we x) \ra (y\we x))\in F
\]
because $x\in F$, so since $1\ra y \in F$
then $x\ra (y\we x)\in F$. Thus, $y\we x\in F$.
Hence, $F$ is a filter.
\end{proof}

In the framework of h-Nelson algebras, as usual we define the binary
term 
\[
x\rla y = (x\ra y)\we (y\ra x).
\]
Let $T\in \KhIL$ and $x,y\in T$. We also define
\[
s(x,y) = (x\rla y) \we (\s x \rla \s y).
\]
It is interesting to note that $s(x,y) = s(y,x) = s(\s x, \s y) = s(\s y, \s x)$.

\begin{lemma}\label{lc6}
Let $T\in \KhIL$ and $x,y,z\in T$. Then the following conditions are satisfied:
\begin{enumerate}[\normalfont 1)]
\item $(x\we (x\rla y)) \ra (y\we (x\rla y)) = 1$,
\item $(x\we s(x,y)) \ra (y\we s(x,y)) = 1$,
\item $(x\we (x\rla y)\we (y\rla z)) \ra (z\we (x\rla y)\we (y\rla z)) = 1$,
\item $(x\we s(x,y) \we s(y,z)) \ra (z\we s(x,y)\we s(y,z)) = 1$,
\item $((x\vee z)\we (x\rla y)) \ra ((y\vee z) \we (x\rla y)) = 1$,
\item $(x\we s(x,y))\ra (z\we s(x,y)) = (y\we s(x,y))\ra (z\we s(x,y))$,
\item $(z\we s(x,y)) \ra (x\we s(x,y)) = (z\we s(x,y)) \ra (y\we s(x,y))$,
\item $((z\we (x\rla y)) \ra (x\we (x\rla y))) \ra ((z\we (x\rla y)) \ra (y\we (x\rla y))) = 1$,
\item $((x\we (x\rla y)) \ra (z\we (x\rla y))) \ra ((y\we (x\rla y)) \ra (z\we (x\rla y))) = 1$,
\item $(x\we \s x) \ra 0 = 0 \ra (x\we \s x) = 1$.
\end{enumerate}
\end{lemma}

\begin{proof}
By Theorem \ref{rept} without loss of generality we can assume that $T = \K(A)$ for
some $A\in \hIL$. Let $x = (a,b)$, $y = (c,d)$ and $z = (e,f)$ be elements of $\K(A)$. 
A direct computation shows that
\[
x\rla y = (a\rla c, (a\we d)\vee (c\we b)),
\]
\[
\s x \rla \s y = (b\rla d, (a\we d)\vee (c\we b)),
\]
\[
y\rla z = (c\rla e, (c\we f) \vee (e\we d)),
\]
\[
\s y \rla \s z = (d\rla f, (c\we f) \vee (e\we d)),
\]
so
\[
s(x,y) = ((a\rla c) \we (b\rla d), (a\we d)\vee (c\we b)).
\]
In order to show 1) note that
\[
x\we (x\rla y) = (a \we (a\rla c), b\vee (a\we d)),
\]
\[
y \we (x\rla y) = (c\we (a\rla c), d\vee (c\we b)).
\]
Since $a \we (a\rla c) = c\we (a\rla c)$,
a direct computation proves that
\[
(x\we (x\rla y)) \ra (y\we (x\rla y)) = (1,0).
\]

The conditions 2) follows from 1) of this lemma and (hN8).

In order to show the condition 3), note that it follows from 1) of this lemma
and (hN8)
that
\[
(x\we (x\rla y)\we (y\rla z)) \ra (y\we (x\rla y)\we (y\rla z) )= 1,
\]
\[
(y\we (x\rla y)\we (y\rla z)) \ra (z\we (x\rla y)\we (y\rla z)) = 1.
\]
Hence, it follows from (hN7) 
that 
\[
(x\we (x\rla y)\we (y\rla z)) \ra (z\we (x\rla y)\we (y\rla z)) = 1.
\]

The condition 4) follows from 3) of this lemma and (hN8).

The condition 5) follows from 1) of this lemma, the distributivity of the
underlying lattice of $T$ and (hN9).

In order to show 6), note that
\[
x\we s(x,y) = (a\we (a\rla c) \we (b\rla d),b \vee (a\we d)),
\]
\[
y\we s(x,y) = (c\we (a\rla c)\we (b\rla d), d\vee (c\we b)),
\]
\[
z\we s(x,y) = (e\we (a\rla c)\we (b\rla d), f \vee (a\we d)\vee (c\we b)).
\]
Given $(u,v) \in T$ we define $\pi_1(u,v) = u$ and $\pi_2(u,v) = v$.
It is immediate that $\pi_1(x\we s(x,y)) = \pi_1(y\we s(x,y))$,
so 
\[
\pi_1((x\we s(x,y)) \ra (z\we s(x,y))) = \pi_1((y\we s(x,y)) \ra (z\we s(x,y))).
\]
Besides, a direct computation shows that 
\[
\pi_2(x\we s(x,y)) \ra (z\we s(x,y)) = f\we a\we (a\rla c)\we (b\rla d),
\] 
\[
\pi_2((y\we s(x,y)) \ra (z\we s(x,y)) = f\we c\we (a\rla c)\we (b\rla d).
\]
Since $a\we (a\rla c) = c\we (a\rla c)$ then 
\[
f\we a\we (a\rla c)\we (b\rla d) = f\we c\we (a\rla c)\we (b\rla d).
\]
Thus, we have proved 6). A similar argument shows 7). 

The conditions 8) and 9) follow from 1) of this lemma and (hN10).

Condition 10) follows from a direct computation.
\end{proof}

Let $T\in \KhIL$. Let us take $F$ an h-implicative filter of $T$. We define the set
\[
\Theta(F) = \{(x,y)\in T\times T:s(x,y)\in F\}.
\]
Note that $s(x,y)\in F$ if and only if $x\ra y, y\ra x, \s x \ra \s y, \s y \ra \s x \in F$.

\begin{lemma}\label{lc7}
Let $T\in \KhIL$ and $F\in \gIF(T)$. Then $\Theta(F)$ is an equivalence relation such that $1/\Theta(F) = F$.
\end{lemma}

\begin{proof}
Let $F\in \gIF(T)$. First we will show that $\Theta(F)$ is an equivalence relation.
It is immediate that $\Theta(F)$ is a reflexive and symmetric relation. In order to show the
transitivity, let $x,y,z\in T$ such that $(x,y), (y,z) \in \Theta(F)$, i.e., $s(x,y), s(y,z) \in F$.
Let $f = s(x,y)\we s(y,z)$. It follows from Lemma \ref{lc5} that $F$ is a filter, so $f\in F$.
Besides it follows from Lemma \ref{lc6} that $(x\we f) \ra (z\we f) = 1$.
But $((x\we f) \ra (z\we f)) \ra (x\ra z)\in F$,  so $x\ra z \in F$.
Similarly we have that $z\ra x\in, \s x \ra \s z, \s z\ra \s x \in F$.
Thus, $s(x,z)\in F$, i.e., $(x,z)\in \Theta(F)$. Hence, $\Theta(F)$ is a transitive relation.
In consequence, $\Theta(F)$ is an equivalence relation.

Finally, let $x\in T$. Then $x\in 1/\Theta(F)$ if and only if $x\ra 1, 1\ra x, 0\ra \s x$ and $\s x\ra 0 \in F$.
Suppose that $x\ra 1, 1\ra x, 0\ra \s x$ and $\s x\ra 0 \in F$. Since $1\ra x\in F$ then it follows from 
Lemma \ref{lc5} that $x\in F$. Conversely, assume that $x\in F$. It follows from Lemma \ref{lc5} that $1\ra x\in F$.
Besides, $((x\we x)\ra (1\we x))\ra (x\ra 1) \in F$, i.e., $1\ra (x\ra 1) \in F$, so $x\ra 1 \in F$.
Moreover, it follows from Lemma \ref{lc6} that $0\ra (\s x \we x) = 1$.
Since $((0\we x) \ra (\s x \we x)) \ra (0\ra \s x)\in F$ then $0\ra \s x \in F$.
Finally, since $(\s x \we x) \ra 0 = 1$ and $((\s x \we x)\ra (0\we x))\ra (\s x \ra 0) \in F$ then
$\s x \ra 0 \in F$. Therefore, $1/\Theta(F) = F$.
\end{proof}

\begin{lemma}\label{lc8}
Let $T\in \KhIL$ and $F\in \gIF(T)$. Then $\Theta(F) \in \Con(T)$.
\end{lemma}

\begin{proof}
First we will see that $\Theta(F)$ preserves $\we$. Let $x,y,z\in T$ such that $(x,y) \in \Theta(F)$, i.e., $s(x,y)\in F$.
It follows from Lemma \ref{lc6} that 
\[
(x\we s(x,y)) \ra (y\we s(x,y)) = 1,
\]
\[
(y\we s(x,y)) \ra (x\we s(x,y)) = 1. 
\]
Thus, it follows from (hN8)
that 
\[
(x\we z\we s(x,y)) \ra (y\we z\we s(x,y)) = 1. 
\]
But 
\[
((x\we z\we s(x,y)) \ra (y\we z\we s(x,y))) \ra ((x\we z) \ra (y\we z))\in F, 
\]
so $(x\we z) \ra (y\we z) \in F$. Similarly, $(y\we z)\ra (x\we z)\in F$.
In order to show that $\s(x\we z) \ra \s(y\we z) \in F$, first 
note that 
\[
(\s x\we s(x,y)) \ra (\s y\we s(x,y)) = 1,
\]
\[
(\s y\we s(x,y)) \ra (\s x\we s(x,y)) = 1,
\]
so it follows from (hN9)
that
\[
(((\s x\we s(x,y)) \vee (\s z \we s(x,y))) \ra((\s y\we s(x,y)) \vee (\s z \we s(x,y))) = 1,
\]
i.e.,
\[
(\s (x\we z) \we s(x,y)) \ra (\s (y\we z)\we s(x,y)) = 1.
\]
Hence, $\s (x\we z) \ra \s(y\we z) \in F$. Analogously, we get 
$\s (y\we z) \ra \s(x\we z) \in F$. Thus, $(x\we z,y\we z) \in \Theta(F)$.

It follows from that $s(x,y) = s(\s x, \s y)$ the fact that $\Theta(F)$ preserves $\s$.
Since we have also proved that $\Theta(F)$ preserves $\we$ and in particular $T$ is a
Kleene algebra, then $\Theta(F)$ preserves $\vee$.

In what follows we will show that $\Theta(F)$ preserves the implication. Let $x,y,z\in T$
such that $(x,y)\in \Theta(F)$, i.e., $s(x,y) \in F$.
Then 
\[
(x\ra z, ((x\we s(x,y))\ra (z\we s(x,y))),
\]
\[
(y\ra z, ((y\we s(x,y))\ra (z\we s(x,y)))
\]
are elements of $\Theta(F)$. Since it follows from Lemma \ref{lc6}
that 
\[
(x\we s(x,y))\ra (z\we s(x,y)) = (y\we s(x,y))\ra (z\we s(x,y))
\]
then the transitivity of $\Theta(F)$ implies that $(x\ra z,y\ra z) \in \Theta(F)$.
A similar argument based on Lemma \ref{lc6} shows that $(z\ra x,z\ra y)\in \Theta(F)$.
Hence, $\Theta(F)$ preserves the implication.

Then, $T/\Theta(F) \in \KhIL$. Therefore, $\Theta(F) \in \Con(T)$.
\end{proof}

\begin{theorem}\label{thm-con}
Let $T\in \KhIL$. The assignments $\theta \mapsto 1/\theta$ and $F\mapsto \Theta(F)$ establish an
order isomorphism between $\Con(T)$ and $\gIF(T)$. 
\end{theorem}

In what follows we introduce and study open implicative filters and N-implicative filters.
These kind of filters will be used later.

\begin{definition}
An implicative filter $F$ of an h-Nelson algebra is said to be \textit{open} if $1\ra x \in F$
whenever $x\in F$.
\end{definition}

From Lemma \ref{lc5}, we obtain the following result.

\begin{proposition}\label{open-if}
Let $T\in \KhIL$. Then every h-implicative filter of $T$ is an open implicative filter of $T$.
\end{proposition}

It is interesting to note that the converse of Proposition \ref{open-if}
is not necessarily satisfied. Indeed, let $T$ be the h-Nelson algebra given in Example \ref{example2}.
Let $F = \{(a,0),(1,0)\}$. It is immediate to show that $F$ is an open
implicative filter. Let $x = (0,1)$, $y = (1,0)$ and $f = (a,0)$.
A direct computation shows that $((x\we f)\ra (y\we f))\ra (x\ra y) = (0,0) \notin F$.
Thus, $F$ is not an h-implicative filter.

Let $\ra_N$ denote the derived implication defined by
\[
x \ra_N y = x \ra (x \wedge y).
\]
We now introduce the following definition.
\begin{definition}
Let $T\in \KhIL$ and $F\subseteq T$. We say that $F$ is an \textit{N-implicative filter}
if $1\in F$ and for every $x,y \in T$ it holds that if $x, x\ra_N \in F$ then $y\in F$.
\end{definition}

\begin{remark} \label{rem-Nsd}
Let $T\in \KhIL$ and $F$ an $N$-implicative filter. Then $F$ is an upset.
Indeed, let $x\leq y$ and $x\in F$. Since $x\ra_N y = 1\in F$ and $x\in F$ then $y\in F$.
Thus, $F$ is an upset.
\end{remark}

\begin{proposition}\label{N-ds}
Let $T\in \KhIL$. Then every h-implicative filter of $T$ is an N-implicative filter of $T$.
\end{proposition}

\begin{proof}
Let $T\in \KhIL$. Let us take $F$ an h-implicative filter of $T$ and $x, x\ra_N y \in F$.
In particular, $x\ra (x\we y) \in F$. Then it follows from F2) that 
$((x\we x) \ra (y\we x)) \ra (x\ra y)\in F$, i.e.,
$(x\ra (x\we y)) \ra (x\ra y) \in F$. Thus, taking into account that $F$ is an implicative filter
and $x, x\ra (x\we y) \in F$ we get $y\in F$. Therefore, $F$ is an 
N-implicative filter of $T$.
\end{proof}

The converse of Proposition \ref{N-ds} is not necessarily satisfied. Indeed,
let $T$ be the the h-Nelson algebra given in Example \ref{example2} and $F = \{(a,0),(1,0)\}$.
We have that $F$ is a not h-implicative filter.
However, a direct computation shows that $F$ is an N-implicative filter.

\section{Centered h-Nelson algebras} \label{s4}

A Kleene algebra (Nelson algebra) is called \textit{centered} if there exists an element which
is a fixed point with respect to the involution, i.e., an element $\ce$ such that 
$\s \ce = \ce$. This element is necessarily unique. If $T = \K(A)$ where $A$ is a bounded distributive
lattice, the center is $\ce = (0,0)$.
Let $T$ be a centered Kleene algebra. We define the following condition:
\begin{center}
$\CK$ For every $x,y\in T$ if $x, y\geq \ce$ and $x\we y\leq \ce$
then there exists $z\in T$ such that $z\vee \ce = x$ and $\s z \vee \ce = y$.
\end{center}
The condition $\CK$ is not necessarily satisfied in every centered Kleene algebra, 
see for instance Figure 1 of \cite{CCSM}. However, every centered Nelson algebra
satisfies the condition $\CK$ (see \cite[Theorem 3.5]{C} and \cite[Proposition 6.1]{CCSM}).

The following two properties are well known: 
\begin{itemize}
\item Let $T$ be a Kleene algebra. Then $T$ is isomorphic to $\K(A)$ for some
bounded distributive lattice $A$ if and only if $T$ is centered and satisfies the
condition $\CK$ (see \cite[Theorem 2.3]{C} and \cite[Proposition 6.1]{CCSM}).
\item Let $T$ be a Nelson algebra. Then $T$ is isomorphic to $\K(A)$ for some
Heyting algebra $A$ if and only if $T$ is centered (see \cite[Theorem 3.7]{C}).
\end{itemize}

An algebra $\langle T,\we,\vee, \ra,\s,0,1,\ce \rangle$ is
a \textit{centered h-Nelson algebra} if $\langle T,\we,\vee, \ra,\s,0,1\rangle$ 
is an h-Nelson algebra and $\ce$ is a center.
We write $\KhILc$ for the variety whose members are centered h-Nelson algebras.

In this section we prove that given $T\in \KhILc$, $T$ is isomorphic to $\K(A)$
for some h-lattice $A$ if and only if $T$ is centered and satisfies the
condition $\CK$ (we also show that the condition $\CK$ is not necessarily satisfied in every
centered h-Nelson algebra). Finally, we show that there exists a categorical equivalence
between $\hIL$ and the full subcategory of $\KhILc$ whose objects satisfy the condition $\CK$
\footnote{If $\mathrm{C}$ is a class of algebras, we abuse notation and also write in this way to refer us
to the algebraic category associated with this class.}. 

We start with some preliminary results.

\begin{remark}\label{rem-c}
\
\begin{enumerate}[\normalfont 1)]
\item Note that if $T, U$ are Kleene algebras, $T$ has a center $\ce_T$ and $f$ is a morphism in $\KhIL$ from $T$ to $U$, then 
$U$ has a center $\ce_U$ and $f(\ce_T) = \ce_U$. Indeed, $f(\ce_T) = f(\sim \ce_T) = \sim f(\ce_T)$.
\item If $T\in \KhILc$ then it follows from Theorem \ref{rept} and 1) that the map
$\rho_T:T\ra \K(T/\co)$ is a monomorphism in $\KhILc$.
\end{enumerate}
\end{remark}

\begin{lemma} \label{ce1}
Let $T\in \KhILc$ and $x,y \in T$.
\begin{enumerate}[\normalfont 1)]
\item If $x\ra y = 1$ then $x\vee \ce \leq y\vee \ce$.
\item If $x\vee \ce = y\vee \ce$ then $x\ra y = 1$.
\item  $\ce \ra 0 = 1$. 
\end{enumerate}
\end{lemma}

\begin{proof}
It follows from Remark \ref{rem-c} that without loss of generality we can assume that $T = \K(A)$ for some $A\in \hIL$.
Let $x = (a,b)$ and $y = (d,e)$ for some $a,b,d,e \in A$ such that $a\we b = d\we e = 0$.

In order to show 1), suppose that $x\ra y = 1$, so $a\leq d$, i.e., $x\vee \ce \leq y\vee \ce$.
Now we will show 2). Suppose that $x\vee \ce = y\vee \ce$, i.e., $a = d$.
Then $x\ra y = (a\ra d,a\we e) =(a\ra a, d\we e) = (1,0)$, which is 2). 
Now, assume that $(x\we y) \ra 0 = 1$, so $(a\we d) \ra 0 =1$. Hence, $a\we d = 0$,
which implies that $0\ra (x\we y) = 1$ and $x\we y \leq \ce$. Finally we will prove 4).
A direct computation shows that $(0,0) \ra (0,1) = (1,0)$.
\end{proof}

Let $T\in \KhILc$. We define the following condition:
\begin{center}
$\CC$ For every $x, y\in T$, if $(x\we y)\ra 0 = 1$ then there exists
$z\in T$ such that $z\vee \ce = x\vee \ce$ and $\s z \vee \ce = y\vee \ce$.
\end{center}

\begin{lemma}\label{ce2}
Let $T\in \KhILc$.
Then $\rho$ is surjective if and only if $T$ satisfies the condition $\CC$.
\end{lemma}

\begin{proof}
Suppose that $\rho$ is surjective. Let $x, y \in T$ such that $(x\we y)\ra 0 = 1$,
so it follows from Lemma \ref{ce1} that $0\ra (x\we y) = 1$.
Thus, $x/\theta \we y/\theta = 0/\theta$. Since $\rho$ is surjective, there exists $z\in T$
such that $z/\theta = x/\theta$ and $\s z/\theta = y/\theta$, i.e., 
$z\ra x = 1$, $x\ra z = 1$, $\s z \ra y = 1$, $y \ra \s z = 1$. Applying
again Lemma \ref{ce1} we get $z\vee \ce = x\vee \ce$ and $\s z \vee \ce = y\vee \ce$.

Conversely, suppose that $\CC$ is satisfied and let $x,y \in T$ such that 
$x/\theta \we y/\theta = 0/\theta$, so $(x\we y)\ra 0 = 1$. It follows from
hypothesis that there exists $z\in T$ such that $z\vee \ce = x\vee \ce$ and 
$\s z \vee \ce = y\vee \ce$. Then it follows from Lemma \ref{ce1} that $z\ra x = 1$, $x\ra z = 1$, 
$\s z \ra y = 1$ and $y\ra \s z = 1$. Then $\rho(z) = (x/\theta, y/\theta)$. 
Therefore, $\rho$ is surjective. 
\end{proof}

\begin{lemma}\label{ce3}
Let $T\in \KhILc$. Conditions $\CC$ and $\CK$ are equivalent.
\end{lemma}

\begin{proof}
Suppose that $T$ satisfies $\CK$ and let $x, y\in T$ such that $(x\we y)\ra 0 = 1$.
It follows from Lemma \ref{ce1} that $x\we y \leq \ce$. Let $\hat{x} = x\vee \ce$ and
$\hat{y} = y\vee \ce$. Then $\hat{x}, \hat{y}\geq \ce$ and $\hat{x}\we \hat{y} \leq \ce$,
so it follows from hypothesis that there exists $z\in T$ such that $z\vee \ce = \hat{x}$ 
and $\s z \vee \ce = \hat{y}$, i.e., $z\vee \ce = x\vee \ce$ and $\s z \vee \ce = y\vee \ce$.

Conversely, suppose that $\CC$ is satisfied. Let $x,y\in T$ such that 
$x, y\geq \ce$ and $x\we y\leq \ce$. It follows that $x\we y=\ce$. It follows from Lemma \ref{ce1} that
$(x\we y)\ra 0 = 1$. Thus, it follows from hypothesis that there exists
$z\in T$ such that $z\vee \ce = x\vee \ce$ and $\s z \vee \ce = y\vee \ce$, i.e.,
$z\vee \ce = x$ and $\s z \vee \ce = y$.
\end{proof}

\begin{theorem}
Let $T\in \KhIL$. Then $T$ is isomorphic to $\K(A)$ for some h-lattice $A$
if and only if $T$ has center and satisfies the condition $\CK$.
\end{theorem}

\begin{proof}
Suppose that $T \cong \K(A)$ for some $A\in \hIL$. Let $x, y \in \K(A)$ such that
$x, y \geq \ce$ and $x\we y \leq \ce$. Thus, there exist $a,b \in A$ such
that $x = (a,0)$, $y = (b,0)$ and $a\we b=0$. The element $z = (a,b) \in \K(A)$ satisfies that  
$z\vee \ce = x$ and $\s z \vee \ce = y$. Hence, $\K(A)$ satisfies $\CK$. Since this condition
is preserved by isomorphisms then $T$ satisfies $\CK$. Furthermore, the fact that $\K(A)$ has center,
which is $(0,0)$, implies that $T$ has center. The converse follows from Lemma \ref{ce2},
Lemma \ref{ce3} and Remark \ref{rem-c}.
\end{proof}

Let $A$ be the Boolean algebra of four elements, where $a$ and $b$
are the atoms. This algebra can be seen as a bounded distributive lattice. Define
a binary operation $\ra$ on $A$ as $x\ra y = 1$ if $x\leq y$ and $x\ra y = 0$ if $x\nleq y$.
It is immediate that $A$ endowed with this binary operation is an h-lattice.
We abuse notation and also write $A$ for this h-lattice.
Then $\K(A)\in \KhILc$.
Let $T$ be the subset of $\K(A)$ given by the following Hasse diagram, which is 
the Hasse diagram of Figure 1 of \cite{CCSM} (see also \cite{LMSM}): 

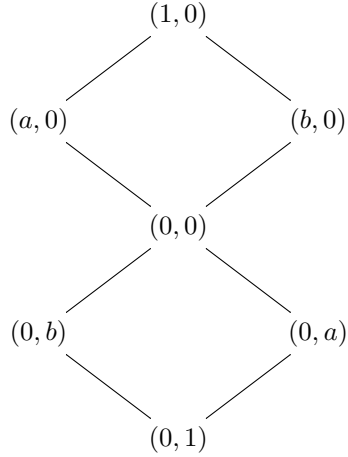
\[
\xymatrix{
   &  (1,0)\\
\ar@{-}[ur] (a,0) &  & \ar@{-}[ul] (b,0)\\
   & \ar@{-}[ul] (0,0) \ar@{-}[ur]\\
   \ar@{-}[ur]  (0,b) &  & \ar@{-}[ul]  (0,a)\\
   & \ar@{-}[ul] (0,1) \ar@{-}[ur]&\\
}
\]

We have that $T$ is a subalgebra of $\K(A)$; hence $T\in \KhILc$. 
Note that $(a,0), (b,0) \geq (0,0)$ and $(a,0)\we (b,0) =  (0,0)$. However, there is
no $z\in T$ such that $z\vee (0,0) = (a,0)$ and ${\sim} z \vee (0,0) = (b,0)$.
Therefore, $T$ does not satisfy $\CK$. It is also interesting to note that $U = \{(0,1),(1,0)\} \in \KhIL$
but $U$ does not have center \footnote{This example was also given in Figure 1 of \cite{LMSM}}. 
\vspace{1pt}

If $A\in \hIL$ then $\K(A) \in \KhIL$. 
Besides, if $f:A\ra B$ is a morphism in $\hIL$ then 
it follows from a direct computation that the map $\K(f): \K(A) \ra \K(B)$ given by $\K(f)(a,b): = (f(a),f(b))$
is a morphism in $\KhIL$. Moreover, $\K$ can be extended to a functor from $\hIL$ to $\KhIL$.  
Conversely, if $T\in \KhIL$ then $\C(T) =: T/\theta \in \hIL$. If $f:T\ra U$ is a morphism in $\KhIL$ then
$\C(f):T/\theta \ra U/\theta$ given by $\C(f)(x/\theta) = f(x)/\theta$ is a morphism in $\hIL$. 
Moreover, $\C$ can be extended to a functor from $\KhIL$ to $\hIL$.

\begin{lemma} \label{alpha}
Let $A\in \hIL$. Then the map $\alpha_{A}:A \ra \C(\K(A))$ given by $\alpha_{A}(a) = (a,a\ra 0)/\theta$
is an isomorphism.
\end{lemma}

\begin{proof}
We write $\alpha$ in place of $\alpha_{A}$.
First we will show that $\alpha$ is a well defined map. Let $a\in A$. Then $a\we (a \ra 0) \leq 0$,
i.e., $a\we (a\ra 0) = 0$. 
Thus, $(a, a\ra 0)/\theta \in \K(A)/\theta$.
Let $a\in A$. Then $(a,a\ra 0) /\theta = (a,0)/\theta$. 
It is immediate that $\alpha$ is a homomorphism. 
The injectivity of $\alpha$ is also immediate. 
In order to show that $\alpha$ is suryective, let $y\in \C(\K(A))$, so $y = (a,b)/\theta$
for some $a, b\in A$ such that $a \we b = 0$. Moreover, $y = (a, a\ra 0)/\theta$,
so $y = \alpha(a)$. Thus, $\alpha$ is suryective. Therefore, $\alpha$ is an isomorphism.
\end{proof}

A direct computation shows that if $f:A \ra B$ is a morphism in $\hIL$ and $a\in A$
then $\C(\K(f))(\alpha_{A}(a)) = \alpha_{B}(f(a))$, and that if $f:T\ra U$ is a morphism in $\KhIL$ and $x\in T$
then $\K(\C(f))(\rho_{T}(x)) = \rho_{U}(f(x))$. 

Therefore, the following result follows from the previous results of this section, 
Remark \ref{rem-c}, and Lemmas \ref{ce2} and \ref{ce3}.

\begin{theorem}\label{center}
There exists a categorical equivalence between $\hIL$ and the full subcategory of $\KhILc$
whose objects satisfy $\CK$.
\end{theorem}

\section{Two important subvarieties} \label{s5}

In \cite{CV} it was generalized the well known relation between Heyting algebras and Nelson algebras 
in the framework of semi-Heyting algebras and semi-Nelson algebras. This relation was also generealized in \cite{LMSM}
for subresiduated lattices and subresiduated Nelson algebras. 
In the present paper we obtain a generalization of the mentioned results for the variety of h-lattices,
which properly contains both the varieties of semi-Heyting algebras and subresiduated lattices.
In this section we recall some definitions and results from \cite{CV,LMSM},
and we study its relation with some of the results of the present paper.
We also describe, in a direct way, the congruences of semi-Nelson algebras 
subresiduated Nelson algebras respectively.

\subsection*{Semi-Nelson algebras}

We start by defining semi-Heyting algebras, which were introduced by 
Sankappanavar \cite{S1} as an abstraction of Heyting algebras. 

\begin{definition}
An algebra $\langle H, \we, \vee, \ra, 0, 1\rangle$ of type $(2,2,2,0,0)$ is a
\textit{semi-Heyting algebra} if the following conditions hold for
every $a,b,d$ in $H$:
\begin{enumerate}[{\rm (S1)}]
\item $\langle H, \we, \vee, 0, 1 \rangle$ is a bounded lattice,
\item $a\we (a\ra b) = a \we b$, 
\item $a\we (b\ra d) = a \we [(a\we b) \ra (a\we d)]$, 
\item $a\ra a = 1$.
\end{enumerate}
\end{definition}

We write $\SH$ for the variety of semi-Heyting algebras.
It is immediate that every semi-Heyting algebra is an h-lattice.
Semi-Heyting algebras share with Heyting algebras the
following properties: they are pseudocomplemented, distributive
lattices and their congruences are determined by the lattice
filters \cite{S1}.

The following definition was given in \cite{CV}.

\begin{definition}
An algebra $\langle T, \wedge, \vee, \to, \sim, 1\rangle$ of type $(2,2,2,1,0)$
is a \textit{semi-Nelson algebra} if for every $x,y,z \in H$ the
following conditions are satisfied:
\begin{enumerate}[{\rm (SN1)}]
 \item $x \we (x \vee y)  =  x$, \label{identidad_absorc}
 \item $x \we (y \vee z)  =  (z \we x) \vee (y \we x)$, \label{identidad_distrib_inf}
 \item $\sim \sim x  =  x$, \label{identidad_doble_neg}
 \item $\sim (x \we y)  =  \sim x \vee \sim y$, \label{identidad_ditrib_neg}
 \item $x \we \sim x  =  (x \we \sim x) \we (y \vee \sim y)$, \label{identidad_kleene}
 \item $x \we (x \to_N  y)  =  x \we (\sim x \vee y)$, \label{identidad_implic_negacion}
 \item $x \to_N  (y \to_N  z)  = (x \we y) \to_N  z$, \label{identidad_implica_infimo}
 \item $(x \to_N  y) \to_N  [(y \to_N  x) \to_N  [(x \to  z) \to_N  (y \to  z)]]  =  1$, \label{identidad_congr_derecha}
 \item $(x \to_N  y) \to_N  [(y \to_N  x) \to_N  [(z \to  x) \to_N  (z \to  y)]]  =  1$, \label{identidad_congr_izquierda}
 \item $(\sim (x \to  y)) \to_N  (x \we \sim y)  =  1$, \label{identidad_negacion_implica_inf1}
 \item $(x \we \sim y) \to_N  (\sim (x \to  y))  =  1$. \label{identidad_negacion_implica_inf2}
\end{enumerate}
where, again, $x \to_N  y$ stands for the term $x \to (x \we y)$.
\end{definition}

Notice that a semi-Nelson algebra is, in particular, a Kleene
algebra, where $0$ is defined as $\s 1$. In what follows we consider
semi-Heyting algebras in the signature $(2,2,2,1,0,0)$.
Semi-Nelson algebras were introduced by Cornejo and 
Viglizzo in \cite{CV} as a generalization of Nelson algebras.
We write $\SN$ for the variety of semi-Nelson algebras.
It follows from \cite[Theorem 4.1]{CV} that if $A\in \SH$ then $\K(A) \in \SN$,
and it follows from \cite[Theorem 3.4]{CV} that if $T\in \SN$ then 
$T/\theta \in \SN$, where $\theta$ is the equivalence relation considered in Lemma \ref{l1}. 
Finally, in \cite[Corollary 5.2]{CV} it was proved that
if $T\in \SN$ then the map $\rho_T$ from 
$T$ to $\K(T/\co)$, defined as in Theorem \ref{rept}, is a monomorphism.
Hence, the following result follows from that every semi-Heyting algebra
is an h-lattice.
              
\begin{proposition}
$\SN\subseteq \KhIL$. 
\end{proposition}

The following property characterizes those h-lattices which are semi-Heyting algebras.

\begin{proposition} \label{prop-semi}
Let $A\in \hIL$. Then $A\in \SH$ if and only if $\K(A) \in \SN$.
\end{proposition}          

\begin{proof}
Let $A\in \hIL$. It is immediate that if $A\in \SH$ then $\K(A) \in \SN$.
Conversely, suppose that $\K(A) \in \SN$, so $\K(A)/\theta \in \SH$, where $\theta$ is
the equivalence relation considered in Lemma \ref{l1}. Then, since it follows from Lemma \ref{alpha}
that $A$ is isomorphic to $\K(A)/\theta$ then $A\in \SH$, which was our aim.
\end{proof}

\begin{remark}
Let $A$ be the h-lattice considered in Example \ref{example1}. Note that
$a\we (a\ra 1) = 0$ and $a\we 1 = 1$, so $a\we (a\ra 1) \neq a\we 1$. 
Hence, $A\notin \SH$. Thus, it follows from Proposition \ref{prop-semi}
that $\K(A) \notin \SN$. Therefore, $\SN$ is properly contained in $\KhIL$. 
\end{remark}

In \cite[Theorem 6.11]{CV} it was proved that if $T\in \SN$ then there exists an order isomorphism
between the lattice of congruences of $T$ and the lattice of N-implicative filters of $T$. Thus,
by Theorem \ref{thm-con} we that the set of N-implicative filters of $T$ coincides with the set
of h-implicative filters of $T$. In what follows we will give some technical results in
order to show in a direct way the last mentioned property.  

\begin{lemma} \label{lem-sh}
Let $A\in \SH$ and $a,b,c,d\in A$.
The following conditions are satisfied:
\begin{enumerate}[\normalfont 1)]
\item If $(a,b), (b,c)\in \K(A)$ then $(a,b)\ra_N (c,d) = (1,0)$ if and only if $a\leq c$.
\item $c\leq (a\ra b) \ra ((a\we c)\ra (b\we c))$.
\item $c\leq ((a\we c)\ra (b\we c)) \ra (a\ra b)$.
\end{enumerate}
\end{lemma}

\begin{proof}
Let $(a,b), (c,d)\in \K(A)$. Note that $(a,b)\ra_N (c,d) = (a\ra (a\we c),a\we d)$.
Thus, $(a,b)\ra_N (c,d) = (1,0)$ if and only if $a\ra (a\we c) = 1$ and $a\we d = 0$.
Taking into account that $a\we (a\ra (a\we c))= a\we c$ and $c\we d = 0$,
we have that $(a,b)\ra_N(c,d) = (1,0)$ if and only if $a\leq c$. Hence, condition 1) is satisfied.
Conditions 2) and 3) follows from \cite[Lemma 9]{SM2} and the fact that the equation 
$1\ra x = x$ is satisfied in every semi-Heyting algebra. 
\end{proof}

\begin{lemma} \label{lem-sn}
Let $T\in \SN$ and $x,y,f \in T$.
The following conditions are satisfied:
\begin{enumerate}[\normalfont 1)]
\item $f\ra_N ((x\ra y)\ra ((x\we f)\ra (y\we f)) = 1$.
\item $f\ra_N ((x\we f)\ra (y\we f))\ra (x\ra y)) = 1$.
\item $f\ra_N (\s (x\ra y)\ra \s ((x\we f)\ra (y\we f)) = 1$.
\item $f\ra_N (\s((x\we f) \ra (y\we f))\ra \sim (x\ra y))=1$.
\end{enumerate}
\end{lemma}

\begin{proof}
Without loss of generality, we can assume that $T = \K(A)$ for some $A\in \SH$.
Conditions 1) and 2) follow from Lemma \ref{lem-sh}.
The proofs of 3) and 4) follow from a direct computation based on the definition
of semi-Heyting algebras.
\end{proof}

\begin{proposition}
Let $T\in \SN$. Then the set of h-implicative filters of $T$ coincides
with the set of N-implicative filters of $T$.
\end{proposition}

\begin{proof}
Let $T\in \SN$. It follows from Proposition \ref{N-ds}
that every h-implicative filter of $T$
is an N-implicative filter of $T$.
The fact that every N-implicative
filter of $T$ is an h-implicative filter of $T$ follows 
from a direct computation based on Remark \ref{rem-Nsd} and 
Lemma \ref{lem-sn}.
\end{proof}

\subsection*{Subresiduated Nelson algebras}

We start with the definition of a subresiduated lattice \cite{EH}.

\begin{definition} 
A \textit{subresiduated lattice} is a pair $(A, D)$, where $A$ 
is a bounded distributive lattice and $D$ is a bounded sublattice of $A$ such 
that for each $a, b \in A$ there exists the maximum of the set $\{d\in D:a\we d\leq b\}$. 
This element is denoted by $a\rightarrow b$.
\end{definition}

Let $(A,D)$ be a subresiduated lattice. This pair can be regarded as an algebra 
$\langle A,\wedge,\vee,\ra,0,1 \rangle$ of type $(2,2,2,0,0)$ where 
$D = \{a \in A : 1 \rightarrow a = a\}=\{1\rightarrow a : a\in A\}$. 
Moreover, an algebra $\langle A,\wedge, \vee, \rightarrow, 0, 1\rangle$ is a subresiduated lattice if and only if 
$\langle A,\wedge, \vee, 0, 1\rangle$ is a bounded distributive lattice and the 
following conditions are satisfied for every $a,b,c\in A$:
\begin{enumerate}[\normalfont 1)]
\item $(a\vee b)\ra c=(a\ra c)\we (b\ra c)$,
\item $c\ra (a\we b)=(c\ra a)\we (c\ra b)$,
\item $(a\ra b)\we (b\ra c)\leq a\ra c$,
\item $a\ra a=1$,
\item $a\we (a \ra b) \leq b$,
\item $a\ra b\leq c\ra (a\ra b)$.
\end{enumerate}

Subresiduated lattices were introduced by Epstein and Horn \cite{EH} as a possible 
generalization of Heyting algebras;
these algebras were also studied by Celani and Jansana as particular cases
of weak Heyting algebras \cite{CJ}. It is immediate that every subresiduated lattice
is an h-lattice.

The following definition was given in \cite{LMSM}.

\begin{definition}
An algebra $\langle T, \we, \vee, \ra, \s, 0, 1\rangle$ of type $(2, 2, 2, 1, 0, 0)$ is said to be a
\textit{subresiduated Nelson algebra} if $\langle T, \wedge, \vee, \sim, 0, 1\rangle$ is a Kleene algebra and the
following conditions are satisfied for every $a,b,c\in T$:
\begin{enumerate}[\normalfont 1)]
\item $(x\vee y)\ra z = (x\ra z)\we (y\ra z)$,
\item $z\ra (x\we y) = (z\ra x)\we (z\ra y)$,
\item $((x\ra y)\we (y\ra z))\ra (x\ra z)=1$, 
\item $x\ra x=1$, \label{imp1}
\item $x\we (x\ra y)\leq x\we (\s x\vee y)$, \label{inf}
\item $x\ra y\leq z\ra (x\ra y)$, \label{impl2}
\item $\s(x\ra y)\ra (x\we \s y)=1$, \label{au}
\item $(x\we \s y)\ra \s(x\ra y)=1$. \label{ultima}
\end{enumerate}
\end{definition}

Subresiduated Nelson algebras were introduced in \cite{LMSM} as a 
generalization of Nelson algebras.
We write $\SNA$ for the variety of subresiduated Nelson algebras.
It follows from \cite[Proposition 3.2]{LMSM} that if $A\in \SRL$ then $\K(A) \in \SNA$,
and it follows from \cite[Proposition 3.6]{LMSM} that if $T\in \SNA$ then 
$T/\theta \in \SRL$, where $\theta$ is the equivalence relation considered in Lemma \ref{l1}. 
Finally, in \cite[Theorem 3.7 5.2]{LMSM} it was proved that
if $T\in \SN$ then the map $\rho_T$ from 
$T$ to $\K(T/\co)$, defined as in Theorem \ref{rept}, is a monomorphism.
In consequence, the following result is consequence from that every subresiduated lattice is
an h-lattice.
              
\begin{proposition}
$\SNA\subseteq \KhIL$. 
\end{proposition}

The following result describes those h-lattices which are subresiduated lattices.

\begin{proposition} \label{prop-sna}
Let $A\in \hIL$. Then $A\in \SRL$ if and only if $\K(A) \in \SNA$.
\end{proposition}          

\begin{proof}
It is similar to the proof of Proposition \ref{prop-semi}.
\end{proof}

\begin{remark}
Let $A$ be the h-lattice considered in Example \ref{example1}. Note that
$(a\vee 0)\ra a =1$ and $(a\ra a) \we (0\ra a) = a$,
so $(a\vee 0)\ra a \neq (a\ra a) \we (0\ra a)$.
Hence, $A\notin \SRL$. Thus, it follows from Proposition \ref{prop-sna}
that $\K(A) \notin \SNA$, so $\SNA$ is properly contained in $\KhIL$. 
\end{remark}

In \cite[Theorem 4.7]{LMSM} it was proved that if $T\in \SNA$ then there exists an order isomorphism
between the lattice of congruences of $T$ and the lattice of open implicative filters of $T$. Thus,
by Theorem \ref{thm-con} we have that the set of open implicative filters of $T$ coincides with the set
of h-implicative filters of $T$. 
In what follows we will give a proof of the last mentioned property avoiding to use
\cite[Theorem 4.7]{LMSM}. We start with the following technical lemma. 

\begin{lemma} \label{oif=gif}
Let $T \in \SNA$ and $x, y,z\in T$. Then the following conditions are satisfied:
\begin{enumerate}[\normalfont 1)]
\item $(x\ra (x\we z)) \ra (((x\we z) \ra (y\we z)) \ra (x\ra y)) = 1$.
\item $(x\ra y) \ra ((x\we z) \ra (y\we z)) = 1$.
\item $((x\we\sim y)\ra \sim ((x\we z)\ra y)) \ra(\sim (x\ra y) \ra \sim ((x\we z)\ra (y\we z)))= 1$.
\item $\sim ((x\we z)\ra (y\we z))\ra \sim (x\ra y) = 1$.
\end{enumerate}
\end{lemma}

\begin{proof}
Let $T \in \SNA$. Without loss of generality we can assume that $T$ 
takes the form $\K(A)$ for some $A\in \SRL$.
In order to show 1), 
it is enough to show that for every $a,b,c \in A$, $a\ra (a\we c) \leq ((a\we c) \ra (b\we c)) \ra (a\ra b)$
\footnote{In \cite[Remark 3.1]{LMSM} it was shown that if $A\in \SRL$ and $(a,b), (c,d)\in \K(A)$,
then $(a,b)\ra (c,d) = (1,0)$ if and only if $a\leq c$.}.
The last equality holds if and only if $(a\ra (a\we c)) \we ((a\we c) \ra (b\we c)) \leq a\ra b$. This follows from the fact that $a\ra (b\we c)\leq a\ra b$. 

That conditions 2), 3) and 4) are satisfied
follows from a direct computation. 
\end{proof}

\begin{remark}
Every open implicative filter of a subresiduated Nelson algebra is an upset.
Indeed, let $T \in \SNA$ and let $F$ be an open implicative filter of $T$.
Let $x,y \in T$ be such that $x \leq y$ and $x \in F$. Then, it is easy to show that $x \ra y = 1$. Since $x,x\ra y \in F$, it follows that $y \in F$.
\end{remark}

\begin{proposition}
Let $T\in \SNA$. Then the set of h-implicative filters of $T$ coincides
with the set of open implicative filters of $T$.
\end{proposition}

\begin{proof}
Let $T\in \SNA$. It follows from Proposition \ref{open-if}
that every h-implicative filter of $T$
is an open implicative filter of $T$. The fact that every open implicative
filter of $T$ satisfies F1) and F4) follows 
from a direct computation based on Lemma \ref{oif=gif}. 

Let $F$ be an open implicative filter of $T$. To prove that it satisfies F2) suppose that $z\in F$. Then, $x\ra (x\we z)=(x\ra x)\we (x\ra z)=x\ra z$. Since $x\leq 1$, we obtain $1\ra z\leq x\ra z$ and from the fact that $1\ra z\in F$, we obtain $x\ra z\in F$. From Lemma \ref{oif=gif}, $(x\ra (x\we z)) \ra (((x\we z) \ra (y\we z)) \ra (x\ra y)) \in F$ and thus $((x\we z) \ra (y\we z)) \ra (x\ra y) \in F$.
To prove that $F$ satisfies F3), suppose that $z\in F$. A direct computation shows that $((x\we \sim y)\ra (x\we \sim y \we z))\ra( (x\we\sim y)\ra \sim ((x\we z)\ra y))=1$. Using an analogous argument, we obtain $x\ra (x\we z)\in F$ and since $x\ra (x\we z)\leq (x\we \sim y)\ra (x\we \sim y\we z) $, it follows that $(x\we \sim y)\ra (x\we \sim y\we z)\in F$. Finally, by Lemma \ref{oif=gif}, $\sim (x\ra y) \ra \sim ((x\we z)\ra (y\we z))\in F$.

Therefore, every open implicative
filter of $T$ is an h-implicative filter of $T$. 
\end{proof}

\subsection*{Acknowledgments}

This work was supported by Consejo Nacional de Investigaciones Cient\'ificas y T\'ecnicas, Argentina (PIP 11220170100195CO, PIP
11220200100912CO), Universidad Nacional de La Plata (11X/921) and Agencia Nacional de Promoci\'on
Cient\'ifica y Tecnol\'ogica (PICT2019-2019-00882, ANPCyT-Argentina). Additional support was provided by the National Science Center (Poland), grant number
2020/39/B/HS1/00216, ``Logico-philosophical foundations of geometry and topology'' and by the MOSAIC project. This last project
has received funding from the European Union’s Horizon 2020 research and innovation programme under the
Sk{\l}odowska-Curie grant agreement No 101007627.

\newpage

-----------------------------------------------------------------------------------
\\
Noem\'i Lubomirsky,\\
Centro de Matem\'atica de La Plata (CMaLP), \\
Facultad de Ciencias Exactas (UNLP), \\
and CONICET.\\
La Plata (1900), Argentina.\\
nlubomirsky@mate.unlp.edu.ar

-----------------------------------------------------------------------------------
\\
Paula Mench\'on,\\
NUCOMPA, Departamento de Matemática,\\
Facultad de Cs. Exactas, \\
Universidad Nacional del Centro de la Provincia de Buenos Aires (UNICEN).\\
Tandil (7000), Argentina.\\
mpmenchon@nucompa.exa.unicen.edu.ar

--------------------------------------------------------------------------------------
\\
Hern\'an Javier San Mart\'in,\\
Centro de Matem\'atica de La Plata (CMaLP), \\
Facultad de Ciencias Exactas (UNLP), \\
and CONICET.\\
La Plata (1900), Argentina.\\
hsanmartin@mate.unlp.edu.ar

\end{document}